\documentclass{amsart}

\usepackage{color}
\usepackage[utf8]{inputenc}

\usepackage{amsmath,amstext,amssymb,mathrsfs,amscd,amsthm}
\usepackage{amsfonts}
\usepackage[all]{xy}
\usepackage{booktabs}
\usepackage{verbatim}
\usepackage{enumerate}
\usepackage[utf8]{inputenc}
\usepackage{graphicx}                   
\usepackage{subfig}
\usepackage{xspace}

\usepackage{amsfonts,latexsym,rawfonts,amsmath,amssymb,amsthm,amscd}
\usepackage{enumerate}
\usepackage{stackrel}
\usepackage{hyperref}
\usepackage{graphicx}
\usepackage{mdwlist}
\usepackage{array} 
\input xy
\xyoption{all}
\usepackage{color}
\usepackage[table]{xcolor}
\usepackage{caption}
\hypersetup{colorlinks, citecolor=black, filecolor=black,linkcolor=black,urlcolor=black} 

\newtheorem{lemma}{Lemma}[section]
\newtheorem{proposition}[lemma]{Proposition}
\newtheorem{theorem}[lemma]{Theorem}
\newtheorem{corollary}[lemma]{Corollary}

\newtheorem{predf}[lemma]{Definition} 
\newenvironment{df}{\begin{predf}\rm}{\end{predf}}
\newtheorem{preremark}[lemma]{Remark}  
\newenvironment{remark}{\begin{preremark}\rm}{\end{preremark}}
\newtheorem{preexample}[lemma]{Example}
\newenvironment{example}{\begin{preexample}\rm}{\end{preexample}}
\newtheorem{prenotation}{Notation}

\newtheorem{prequestion}[lemma]{Question}
\newenvironment{question}{\begin{prequestion}\rm}{\end{prequestion}}
\newtheorem{preanswer}[lemma]{Answer}

\theoremstyle{definition}

\newcommand{\Img}{\mathrm{Im }}
\newcommand{\CC}{\mathbb{C}}
\newcommand{\QQ}{\mathbb{Q}}
\newcommand{\ZZ}{\mathbb{Z}}
\newcommand{\Mm}{\mathcal{M}}

\def\ori{\mathfrak{o}}
\def\sing{\mathrm{sing}}

\def\Aut{\mathrm{Aut}}
\def\Top{\mathbf{Top}}

\def\QCDGA{\bQ\text{-}\mathbf{CDGA}}
\def\Apl{\mathcal{A}_{PL}}
\def\tildeApl{\tilde{A}_{PL}}
\def\Gr{\mathrm{Gr}}

\newcommand\mmnote[1]{}

\newcommand{\bL}{\mathbb{L}}

\newcommand{\bP}{\mathbb{P}}
\newcommand{\bQ}{\mathbb{Q}}
\newcommand{\bR}{\mathbb{R}}

\newcommand{\bZ}{\mathbb{Z}}

\newcommand\lra{\longrightarrow}

\newcommand\Th{\mathrm{Th}}

\newcommand\Ker{\operatorname*{Ker}}

\newcommand{\pb}{\ar@{}[dr]|{\mbox{\LARGE{$\lrcorner$}}}}

\def\CP{\mathbb{C}P}

\def\haut{\mathrm{hAut}_c}

\renewcommand{\geq}{\geqslant}
\renewcommand{\leq}{\leqslant}
\newcommand{\cA}{\mathcal{A}}

\newcommand{\cM}{\mathcal{M}}

\usepackage{color}

\def\haut{\mathrm{haut}}

\def\SO{\mathrm{SO}}

\AtBeginDocument{%
   \def\MR#1{}
} 

\title[Weight decompositions of Thom spaces in rational homotopy]{Weight decompositions of Thom spaces of vector bundles in rational homotopy}
\author{Urtzi Buijs}
\author{Federico Cantero Morán}
\author{Joana Cirici}
\thanks{The first author has been partially supported by the Ramón y Cajal MINECO programme, by the Junta de Andalucía grant 
FQM-213 and by the MINECO grant MTM2013-41768-P. The second author was funded by the Belgian Interuniversity Attraction Pole P07/18 and by the ``María de Maeztu'' MINECO grant MTM2014-0445 through the BGSMath. The third author was funded by the German Research Foundation (SPP 1786). F. Cantero and J. Cirici were partially supported by the MINECO grant MTM2013-42178-P}
\keywords{Thom spaces, rational homotopy theory, smoothing theory, Thom isomorphism theorem, mixed Hodge structures, weight decompositions, motivic Thom spaces}

\begin{document}
\begin{abstract}
Motivated by the theory of representability classes by submanifolds, we study the rational homotopy theory of Thom spaces of vector bundles.
We first give a Thom isomorphism at the level of rational homotopy, extending work of Felix-Oprea-Tanré by
removing hypothesis of nilpotency of the base and orientability of the bundle.
Then, we use the theory of weight decompositions in rational homotopy to give a criterion of 
representability of classes by submanifolds, generalising results of Papadima.
Along the way, we study issues of formality and give formulas for Massey products of Thom spaces. 
Lastly, we link the theory of weight decompositions with mixed Hodge theory
and apply our results to motivic Thom spaces.
\end{abstract}

\maketitle

\section{Introduction}
Let $\xi\colon E\to B$ be a vector bundle of rank $n$. After endowing it with a Riemannian metric, one can consider
its unit sphere bundle $S(\xi)\colon S(E)\to B$ and its unit disc bundle $D(\xi)\colon D(E)\to B$. The \emph{Thom space} of 
the bundle $\xi$, denoted $\Th(\xi)$, is the result of collapsing the subspace $S(E)$ of $D(E)$ to a point. Thom spaces are
fundamental in differential topology, as they parametrise cobordism classes of manifolds with certain structures, and in
algebraic topology, being used to define a rich class of generalised cohomology theories. 
In particular, Thom spaces play an important role in the theory of of representability of cohomology classes by submanifolds.
As noted by Papadima in \cite{Papadima}, in certain cases, the question of representability of a cohomology class by a submanifolds
may be translated to a question in rational homotopy theory.

In the sixties, Sullivan and Quillen developed a formalism
that allows to associate to each
topological space $X$ another space $X_\bQ$ (all of whose homotopy
groups are rational vector spaces) and a map $X\to X_{\bQ}$ inducing isomorphisms on rational cohomology. This new 
space captures all the rational information of $X$, the \emph{rational homotopy type of $X$}. 
Sullivan further constructed a functor $\cA_{PL}$ from topological spaces to commutative differential graded algebras (cdga's)
and showed that, when restricted to nilpotent spaces,
this functor is an equivalence after inverting rational 
homology equivalences of spaces and quasi-isomorphisms of cdga's (\cite{Sullivan}, see also \cite{FHT}).
A space $X$ is said to be \emph{formal} if there is a string of quasi-isomorphisms from
$\cA_{PL}(X)$ to the cdga with trivial differential given by the rational cohomology $H^*(X,\bQ)$ of $X$.
A space is called \emph{intrinsically formal} if it is rationally homotopy equivalent
to any other simply connected space with the same rational cohomology ring.

Since Thom spaces of vector bundles of rank $n$ are simply connected as long as $n>1$,
the rational localisation theory of Quillen and Sullivan applies, 
even if the base of the vector bundle is not nilpotent or the vector bundle is non-orientable.
The rational homotopy type of Thom spaces of universal vector bundles over $BGL(\bR^n)$ was studied by Burlet \cite{Burlet}. 
Later, Papadima \cite{Papadima} showed that Thom spaces of orientable vector bundles over classifying spaces of closed connected subgroups of $GL(\bR^n)$ are formal.
As an application, he studied the existence of maps $\phi\colon M\to \Th(\gamma_{d,\infty})_\bQ$ from $M$ 
to the rationalisation of $\Th(\gamma_{d,\infty})$ such that $\phi^*(u)=c$, where $u$ denotes the Thom class.
Then, he used a result on derationalisation of maps of topological spaces whose target is formal,
to establish the conditions for a multiple of a given cohomology class $c$ of a manifold $M$ to be representable by a submanifold.

In Section $\ref{S2}$ of this paper, we describe the rational homotopy type of the Thom space of any vector bundle $\xi:E\lra B$
in terms of the rational homotopy type of $B$ and the Euler class of $\xi$. 
This allows us to formulate, in Theorem $\ref{mainthmThom}$, a Thom isomorphism theorem at the level of rational homotopy.
This was recently obtained  by Félix, Oprea and Tanré \cite{FOT16} in the case of oriented vector bundles with nilpotent base. 
Our more conceptual approach allows us to remove these hypotheses. In Theorem $\ref{mainthmThomtwisted}$ we give a 
version for non-oriented vector bundles with twisted coefficients.

Section $\ref{S3}$ is devoted to the theory of weight decompositions.
We say that a space $X$ \textit{admits a (positive) weight decomposition} if there
exists a certain (positive) bigrading on a cdga $\Mm$ compatible with products and differentials
(see Definition $\ref{defweightdec}$ for details),
together with a quasi-isomorphism $\cM\lra \cA_{PL}(X)$.
The formality of $X$ can be reformulated in terms of the existence of a
weight decomposition with pure weights. Therefore one can think of weight
decompositions as an intermediate property towards formality.
Furthermore, as indicated to us by Yves Felix, the existence of positive weight decompositions turns out to be sufficient
to derationalise: given CW-complexes $X$ and $Y$ satisfying certain conditions, every map $X_\bQ\to Y_\bQ$ can be lifted to 
a map $X\to Y$ provided $Y$ admits a positive weight decomposition (Theorem \ref{derationalise}). 
Together with our description of the rational homotopy type of Thom spaces, this leads to a criterion of representability of 
cohomology classes by $\theta$-submanifolds, where $\theta\colon B\to BSO(q)$ is a fibration,
in terms of the existence of positive weight decompositions, thus extending Papadima's work on this problem (see Theorem $\ref{representability}$).

Lastly, in Section $\ref{S4}$ we study algebraic vector bundles.
The rational cohomology of every complex algebraic variety $X$ carries functorial mixed Hodge structures.
In fact, these descend to  mixed Hodge structures on the rational homotopy type of $X$.
In Theorem $\ref{AlgVarsWeightDec}$, we show that mixed Hodge theory leads to functorial weight decomposition.
Furthermore, for smooth quasi-projective varieties, such weight decompoisions are always positive.
Let $\xi:E\to B$ is be
complex vector bundle over a smooth complex variety $B$.
The collapsing construction needed to define a Thom space is not algebraic, 
and therefore $\Th(\xi)$ is not, in general, a complex algebraic variety. 
Nonetheless, it is motivic space, and we study it as such.
In Theorem \ref{MHSenelmodelo}, we describe the mixed Hodge structures on the rational homotopy type of motivic Thom spaces.
As a consequence, we are able to extend Papadima's 
result whenever
$\theta$ is the classifying map of the underlying real vector bundle of an algebraic vector bundle over an algebraic variety.

\subsection*{Acknowledgements} We thank Yves F\'elix for answering our many questions and for telling us about spaces with weight decompositions. Thanks also to Vicente Navarro for his ideas on mixed Hodge theory and useful comments.
We are also grateful to Andrew Baker, Pascal Lambrechts, Luc Menichi and Jean-Claude Thomas for their feedback at the early stages of this project.

\section{A Thom isomorphism theorem in rational homotopy}\label{S2}
In this section, we study the rational homotopy type of Thom spaces. 
After recalling some preliminaries on Thom's isomorphism in cohomology, 
we first study the simpler case of vector bundles of odd rank.
In this case, we show that Thom spaces are formal.
For the even rank case, we give a model of the Thom space from a model of the base and the Euler class of the vector bundle.
This leads to a Thom isomorphism theorem at the level of rational homotopy.
Then, we extend our study to the case of non-orientable vector bundles.

\subsection{Preliminaries: Thom isomorphism in cohomology}
The cohomology with coefficients in a ring $R$ of the Thom space of an oriented vector bundle $\xi:E\to B$ is well-understood: by excision,
we have that 
$$\tilde{H}^*(\Th(\xi);R)\cong H^*(D(E),S(E);R)$$
and as $D(\xi)\colon D(E)\to B$ is a homotopy equivalence, $H^*(D(E);R)\cong H^*(B;R)$.
The relative cup product then gives a pairing
$$H^*(B;R)\otimes \tilde{H}^*(\Th(\xi);R)\to \tilde{H}^*(\Th(\xi);R).$$
\begin{theorem}[Thom isomorphism]
If $\xi$ is an oriented vector bundle of rank $n$, then there is a class $u\in H^n(\Th(\xi);R)$, called the \emph{Thom class}, such that the homomorphism 
\[-\cup u\colon H^*(B;R)\to \tilde{H}^{*+n}(\Th(\xi);R)\]
given by relative cup product is an isomorphism of graded $R$-modules.
\end{theorem}

The composite of the zero section of $\xi$ and the quotient map $D(E)\to \Th(\xi)$ 
is the \emph{canonical inclusion}, and will be denoted by $\iota$. One has that 
$\iota^*(u)=e$, the Euler class of $\xi$. Using the naturality of the relative 
cup product for $\iota$, it follows that $\iota^*(a\cup u) = a\cup e$ and that 
\[(a\cup u)\cdot (b\cup u) = (a\cdot b \cdot e)\cup u.\] Therefore, the Thom 
isomorphism theorem gives a complete description of the cohomology ring of 
the Thom space of any orientable vector bundle.

If $\xi:E\lra B$ is a non-orientable vector bundle, there is a twisted version of Thom's isomorphism by 
considering the orientation bundle $\ori$ and noting that
 $\ori\otimes\ori$ is the trivial line bundle (see for instance Theorem 7.10 of \cite{BottTu}). 
\begin{theorem}[Thom isomorphism with twisted coefficients]
If $\xi$ is a vector bundle of rank $n$, then there is a class $u\in H^n(D(E),S(E);\ori)$, such that the homomorphism 
\[-\cup u\colon H^*(B;\ori)\to \tilde{H}^{*+n}(\Th(\xi))\]
given by relative cup product is an isomorphism of graded $R$-modules. In this case, 
\[(a\cup u)\cdot (b\cup u) = (a\cdot(b\cdot e))\cup u.\]
\end{theorem}

\subsection{Vector bundles of odd rank}
We begin by studying the simpler case of vector bundles of odd rank and show that the resulting Thom spaces are formal.

In order to understand the homotopy type of Thom spaces, we divide the collapsing in the definition of a Thom space into two steps: First, for a vector bundle $\xi\colon E\to B$ of odd rank $n$, one can consider its fibrewise one-point compactification $\dot{\xi}\colon \dot{E}\to B$, which is a $S^n$-bundle, that can also be obtained as the fibrewise unreduced suspension of the unit sphere bundle $\xi\colon S(E)\to B$. This bundle comes with a canonical section $s\colon B\to \dot{E}$, namely the section that sends each point $p$ in $B$ to the point at infinity in the fibre over $p$. Then, the Thom space of $\xi$ is homeomorphic to the result of collapsing the image of $s$ inside $\dot{E}$:
\[B\overset{s}{\lra} \dot{E}\lra \Th(\xi).\]

Let $\haut_*(S^n)$ be the monoid of pointed homotopy automorphisms. Isomorphism classes of pairs $(\eta,s)$ where $\eta$ is a $S^n$-fibration over $B$ and $s$ is a section of $\eta$ are classified by homotopy classes of maps $B\to B\haut_*(S^n)$. If $n$ is odd, then $\haut_*(S^n)$ has two rationally contractible components, corresponding to orientable and unorientable bundles. Therefore, in this case, the map $s\colon B\to \dot{E}$ is rationally homotopy equivalent to the map $s'\colon B\to \dot{P}$, where $\psi\colon P\cong L\times \bR^{n-1}\to B$ is the Whitney sum of some line bundle $\ell\colon L\to B$ and a trivial vector bundle of rank $n-1$. But then we have that $\Th(\xi)\cong \Th(\psi) \simeq \Sigma^{n-1}\Th(\ell)$. Since suspensions are formal, we have that if $\ell$ is orientable or $n\geq 3$, then $\Th(\xi)$ is formal, and therefore the cdga $H^*(\Th(\xi);\bQ)$ with trivial differential is a cdga model of $\Th(\xi)$. In the remaining case $\Th(\xi)$ fails to be simply-connected, even nilpotent in general, so the localisation theory does not apply in this case.

\subsection{Oriented vector bundles of even rank}
Let $Y\subset X$ be a cofibration. Then it induces a surjection 
$\varphi:\Apl(X)\to \Apl(Y)$ and we may use the model of the cofibre to write
$$\Apl(X,Y)\simeq \Ker(\varphi).$$ 
Denote by $\iota^*:\Apl(X,Y)\hookrightarrow \Apl(X)$ the resulting injective morphism.
Since the kernel is an ideal, we have a relative cup product 
at the level of cdga's.

\begin{lemma}\label{lemma:kernel}
The assignment $(x,y)\mapsto x\cdot \iota^*(y)$ defines a product
$$\cup:\Apl(X)\otimes \Apl(X,Y)\lra \Apl(X,Y)$$
which induces the relative cup product in cohomology.
\end{lemma}
\begin{proof}
It suffices to note that for every $x\in \Apl(X)$ and $y\in \Apl(X,Y)\subset\Apl(X)$, we have 
$x\cdot \iota^*(y)\in \Apl(X,Y)\subset \Apl(X)$.
\end{proof}

Let $\xi:E\lra B$ be an oriented vector bundle of even rank $n$.
Denote by
$$\tilde \cA_{PL}(\Th(\xi)):=\cA_{PL}(\Th(\xi),\ast)\simeq \cA_{PL}(D(E),S(E))$$
the algebra of piece-wise linear forms of the inclusion $\ast\lra \Th(\xi)$.
The cohomology of $\tilde \cA_{PL}(\Th(\xi))$ computes the reduced cohomology of the Thom space.
To get the model of $\Th(\xi)$ one just needs to add $\bQ$ in degree 0:
$$\cA_{PL}(\Th(\xi))^0= \bQ\text{ and }
\cA_{PL}(\Th(\xi))^k= \tilde\cA_{PL}(\Th(\xi))^k\text{ for }k>0.$$

By Lemma \ref{lemma:kernel}, we have a relative cup product at the level of cdga's
$$\cup\colon \Apl(D(E))\otimes \Apl(D(E),S(E))\lra \Apl(D(E),S(E)).$$

\begin{lemma}\label{productemagic}Let $u\in \cA_{PL}(D(E),S(E))$ be a representative of the Thom class 
of $\xi$ and let $e=\iota^*(u)$.
Then for every $x,y\in\cA_{PL}(D(E))$ we have
$$\iota^*(x\cup u)=x\cdot e\text{ and }(x\cup u)\cdot (y\cup u)=(x\cdot y\cdot e)\cup u.$$
\end{lemma}
\begin{proof}
The first equality is straightforward using the definition of $\iota^*$ and $\cup$. Indeed, we have
$$\iota^*(x\cup u)=x\cdot \iota^*(u) = x\cdot e.$$
Let us prove the second equality, using Lemma \ref{lemma:kernel}. We have
\begin{align*}
(x\cdot y\cdot e)\cup u=(x\cdot y\cdot e)\cdot \iota^*(u)= x\cdot y\cdot e\cdot e = (x\cdot e)\cdot (x\cdot e) = \\ =(x\cdot \iota^*(u))\cdot (y\cdot \iota^*(u)) 
= (x\cup u)\cdot (y\cup u).
\end{align*}
where we used the fact that $e$ has even degree. 
\end{proof}

Before stating the main result of this section, let us fix some notation.
Let $n\geq 0$ be an integer. The \textit{$n$-th suspension} of a graded vector space
$\cA=\bigoplus \cA^k$ is the graded vector space $s^n\cA$ defined by 
$s^n\cA^k:=\cA^{k-n}$. If $a\in \cA$ is an element of $\cA$,
we will denote by $w_a:=s^na\in s^n\cA$ its $n$-th suspension.

\begin{df}\label{defM} Let $\cA$ be a cdga and let $e\in \cA^n$, with $n$ even. Define $\cA[e]$ as 
the graded vector space given by the $n$-th suspension $s^n\cA$ of $\cA$. Define a differential $d:\cA[e]^k\to \cA[e]^{k+1}$ and a product $\mu:\cA[e]^k\otimes \cA[e]^l\to \cA[e]^{k+l}$ by setting $d(w_x) = w_{dx}$ and $\mu(w_x,w_y) = w_{exy}$ respectively.
\end{df}

\begin{proposition}\label{mainpropThom}
Let $\xi\colon E\to B$ be an oriented vector bundle of even rank $n$. Let $u\in \Apl(D(E),S(E))$ be a representative of the Thom class, so that $e:= \iota^*(u)$ is a representative of the Euler class. Denote by $\cA$ the cdga $\Apl(D(E))$.
\begin{enumerate}
\item The morphism $g\colon \cA[e]\to \Apl(D(E),S(E))$ given by $g(w_x) = x\cup u$  is a quasi-isomorphism.
\item The product $\cup \colon \cA\otimes \cA[e]\lra \cA[e]$ given by 
$(x, w_y)\mapsto w_{xy}$ is a model of the relative cup product, that is, the following diagram commutes:
\[
\xymatrix{
\cA\otimes \cA[e]\ar[d]^{f\otimes g}\ar[r]^-{\cup}&\cA[e]\ar[d]^g\\
\cA\otimes \cA_{PL}(D(E),S(E))\ar[r]^-{\cup}&\cA_{PL}(D(E),S(E))
}
\]
\end{enumerate}
\end{proposition}
\begin{proof}
The second statement is immediate from the first statement and Lemma \ref{lemma:kernel}. 
For the first statement, since $d{u'}=0$ it is clear that $g$ is compatible with differentials. To see that $g$ is multiplicative it suffices to show that 
$$(x\cup u)\cdot (y\cup u)=(x\cdot y\cdot e)\cup u,$$
for every $x,y\in\cA$. This follows from Lemma $\ref{productemagic}$, since $e=\iota^*(u)$.
Therefore $g$ is a morphism of cdga's. 
Additionally, it induces the Thom isomorphism after taking cohomology, 
hence, by the Thom isomorphism theorem, $g$ is a quasi-isomorphism.
\end{proof}

\begin{lemma}\label{lemma:thom1}
The triple $(\cA[e],d,\mu)$ of Definition $\ref{defM}$ is a cdga
which does not depend, up to quasi-isomorphism, on the chosen representative of the Euler class.
\end{lemma}
\begin{proof}
Since by assumption, $e$ is of even degree, 
for every $a\in \cA$, we have $(-1)^{|w_a|}=(-1)^{|a|}$.
Hence $\cA[e]=s^n\cA$ with the product $w_a\cdot w_b=w_{e a b}$
is graded commutative. Also, since $de=0$, we have
$$d(w_a\cdot w_b)=d(w_{e a b})=w_{e d(a b)}=w_{e(da\cdot b+(-1)^{|a|}a\cdot db)}=
w_{da}\cdot w_b+(-1)^{|w_a|}w_a\cdot w_{db}.
$$
Therefore the Leibniz rule is satisfied. This proves that $(\cA[e],d,\mu)$ is a cdga.

Let $e,e'$ be two representatives of the Euler class
. Then, $e-e'$ is exact, so there is a $z$ with $dz=e-e'$. Define a homotopy between the multiplications in $\cA[e]$ and $\cA[e']$ as $H(x\otimes y) = xyz$.
\end{proof}

The following lemma is straightforward:
\begin{lemma}\label{lemma:thom2} Let $f\colon \cA\to \cA'$ be a quasi-isomorphisms of cdgas, and let $e\in \cA$ be a class of even dimension. Then
\begin{enumerate}
\item The morphism $g\colon \cA[e]\to \cA'[f(e)]$ given by $g(w_x) = w_{f(x)}$ is well-defined and a quasi-isomorphism.
\item The following diagram commutes
\[
\xymatrix{
\cA\otimes \cA[e]\ar[d]^{f\otimes g}\ar[r]^-{\cup}&\cA[e]\ar[d]^{g}\\
\cA'\otimes \cA'[f(e)]\ar[r]^-{\cup}&\cA'[f(e)]
}
\]
\end{enumerate}
\end{lemma}

Puting together Proposition \ref{mainpropThom} and the last two lemmas, we have:
\begin{theorem}\label{mainthmThom}
Let $\xi\colon E\to B$ be an oriented vector bundle of even rank $n$, let $\cA$ be a cdga model of $B$ and let $e\in \cA^n$ be a representative of the Euler class. Then $\cA[e]$ is a cdga model of $\Th(\xi)$ and the product $\cA\otimes \cA[e]\to \cA[e]$ given by $(x,w_y)\mapsto w_{xy}$ is a model of the relative cup product.
\end{theorem}

\subsection{Non-orientable vector bundles}

Note that if $Y\subset X$ is a cofibration of path connected spaces, then
taking local coefficients in a pair of real line bundles $\ell,\ell'$ over $X$ gives the relative cup product with local coefficients:
\[H^p(X;\ell)\otimes \tilde{H}^q(X,Y;\ell')\lra \tilde{H}^{p+q}(X,Y;\ell\otimes\ell').\]

Recall the functor $\Apl\colon \Top\to \QCDGA$ that first takes a space to its singular simplicial 
set $\sing(X)$, and then takes piecewise linear forms on this simplicial set. In detail, let $A$ be the simplicial cdga given by 
\[A(n) = \Lambda\left(t_0,t_1,\ldots,t_n,dt_0,dt_1,\ldots,dt_n\right)/\left(\sum_{i=0}^n{t_i} = 1, \sum_{i=0}^n dt_i = 0\right)\]
with the obvious face and degeneracy maps. Then $\Apl(X)$ is the set of simplicial maps from $\sing(X)$ to $A$. The cdga structure on $A$ induces a cdga structure on 
the piecewise linear forms $\Apl(X)$. 

This functor generalises further to the category of pairs of topological spaces with local coefficient systems, i.e., the category $\Top^\tau$ of triples $(X,Y,\ell)$, where $Y\subset X$ is a cofibration of topological spaces and $\ell\colon E\to X$ is a real line bundle. There is a homomorphism $\varphi\colon \mathrm{O}(\bR)\cong \bZ/2\to \Aut(A(n))$ that sends the reflection to the automorphism that interchanges $t_0$ with $t_1$, and $dt_0$ with $dt_1$ and leaves $t_i$ unchanged for $i>1$. Using $\varphi$, we can associate to $(X,\ell)$ a new bundle $(X,\bar{\ell})$ of simplicial cdga's over $X$ whose fibre is isomorphic to $A$. Taking singular simplices, we obtain a bundle of simplicial cdga's over $\sing(X)$ with fibre isomorphic to $A$. Define $\Apl(X,Y;\ell)$ as the set of simplicial sections of this bundle that take value $0$ on $\sing(Y)\subset \sing(X)$. 

The set $\Apl(X,Y;\ell)$ has the structure of a differential graded $\bQ$-vector space, and, if $\ell'$ is another line bundle over $X$, the diagonal map induces multiplications
\begin{align}\label{eq:17}
\Apl(X,Y;\ell)\otimes\Apl(X,Y;\ell')&\lra \Apl(X,Y;\ell\otimes\ell')\\
\Apl(X;\ell)\otimes \Apl(X,Y;\ell')&\lra \tildeApl(X,Y;\ell\otimes \ell')
\end{align}
If $\ell$ is trivial, this becomes the usual multiplication and $\Apl(X,Y;\ell)$ is isomorphic to the kernel of the augmentation $\Apl(X/Y)\to \Apl(Y/Y)\cong \bQ$. 
This multiplication is natural with respect to the map $(X,\emptyset)\to (X,Y)$, i.e., the following diagram commutes:
\[\xymatrix{
\Apl(X;\ell)\otimes\Apl(X,Y;\ell')\ar[d]\ar[r]& \Apl(X,Y;\ell\otimes\ell')\ar[d]\\
\Apl(X;\ell)\otimes \Apl(X;\ell')\ar[r]& \Apl(X;\ell\otimes \ell')
}\]
Observe additionally that the differential graded module $\Apl(X,Y;\ell\otimes\ell')$ includes in
$\Apl(X,\ell\otimes\ell')$ as the kernel of the restriction homomorphism 
$\Apl(X;\ell\otimes \ell')\to \Apl(Y;i^*(\ell\otimes\ell'))$, and therefore the rightmost vertical map is injective. 
In particular, Lemma $\ref{lemma:kernel}$ is also true for the relative cup product with twisted coefficients.

Let $\xi\colon E\to B$ be a vector bundle of rank $k$. Let $\ori$ be the orientation bundle of $\xi$ 
and recall that $\ori\otimes\ori$ is the trivial line bundle.
By the previous discussion the relative cup product 
\[H^p(D(E);\ori)\otimes H^q(D(E),S(E);\ori)\lra H^{p+q}(D(E),S(E))\]
is induced by
$$\cup\colon \Apl(D(E);\ori)\otimes \Apl(D(E),S(E);\ori)\lra \Apl(D(E),S(E)).$$
It is straightforward to verify that the formulas of Lemma $\ref{productemagic}$ are also valid in the twisted setting,
for any $x,y\in\cA_{PL}(D(E);\ori)$.

Now, let $\xi\colon E\to B$ be a vector bundle of rank $n$, and let $\cA$ be a cdga model of $\Apl(D(E))$, and let $\cA_\ori$ be a cdga model of $\Apl(D(E);\ori)$, so that \eqref{eq:17} gives the following products:
\[\cA\otimes \cA\lra \cA,\quad \cA\otimes\cA_\ori\lra\cA_\ori, \quad \cA_\ori\otimes \cA_\ori\lra \cA.\]
Let $e\in \cA_\ori$ be a representative of the twisted Euler class.

\begin{df} In this setting, define the cdga $\cA_\ori[e]$ whose underlying vector space is $s^n\cA_\ori$, with differential $d(w_x)=w_{dx}$ and multiplication $\mu(w_x,w_y)=w_{exy}$. 
\end{df}

\begin{proposition}\label{mainpropThomtwisted} Let $\cA = \Apl(D(E))$ and let $\cA_\ori = \Apl(D(E);\ori)$. Let $u\in \Apl(D(E),S(E);\ori)$ be a representative of the Euler class, and let $e=\iota^*(u)\in \cA_\ori$ be its pullback. Then: 
\begin{enumerate}
\item The homomorphism $g\colon \cA_\ori[e]\to \Apl(D(E),S(E))$ given by $g(w_x) = x\cup u$ is a quasi-isomorphism. 
\item The product $\cup \colon \cA_\ori\otimes s^n\cA\lra \cA_\ori[e]$ given by 
$(x, w_y)\mapsto w_{xy}$ is a model of the relative cup product, that is, the following diagram commutes:
\[
\xymatrix{
\cA_\ori\otimes s^n\cA\ar[d]^{f\otimes g}\ar[r]^-{\cup}&\cA_\ori[e]\ar[d]^g\\
\cA_{PL}(D(E);\ori)\otimes \cA_{PL}(D(E),S(E);\ori)\ar[r]^-{\cup}&\cA_{PL}(D(E),S(E))
}
\]
\end{enumerate}
\end{proposition}
\begin{proof}
It is straightforward to verify that the proof of Theorem $\ref{mainthmThom}$ is valid in the twisted setting with the obvious modifications
and using the twisted version of Thom's isomorphism theorem in cohomology.
\end{proof}
Using twisted versions of Lemmas \ref{lemma:thom1} and \ref{lemma:thom2}, we arrive to the most general version of the Thom isomorphism theorem at the level of cochains:

\begin{theorem}\label{mainthmThomtwisted} Let $\xi\colon E\to B$ be a vector bundle of rank $n$ and let $\ori$ be the orientation bundle of $E$. Let $\cA$ and $\cA_\ori$ be cdga models of $B$ with trivial and twisted coefficients respectively, and let $e\in \cA_\ori$ be a representative of the twisted Euler class. Then the cdga $\cA_\ori[e]$ is a model of $\Th(\xi)$ and $\cA_\ori\otimes s^n\cA\to \cA_\ori[e]$ is a model of the relative cup product.
\end{theorem}

\subsection{Lie models of Thom spaces} 
In the algebrisation of rational homotopy theory of Quillen \cite{Quillen}, rational spaces are modelled by differential graded Lie algebras (dgl's). Each quasi-isomorphism class of dgl's contains a canonical free dgl that is unique up to isomorphism. If $X$ is a simply connected space, then the canonical dgl that models $X$ receives the name of \emph{Quillen model of $X$}, and its vector space of generators is canonically identified with the rational homology of $X$ (here we assume that $X$ is a finite complex). 

In our situation, the Thom isomorphism theorem gives us the basis of the Quillen model for free, and therefore it only remains to figure out its differential. 

A natural way to try to do this is to use the homotopy transfer theorem (HTT) \cite{Loday-Vallette}, which establishes a correspondence between isomorphism classes of algebra structures on a chain complex and isomorphism classes of $A_\infty$-structures on the homology of the chain complex. Then, an $A_\infty$-structure on the cohomology gives a canonical differential on the free Lie algebra generated by the desuspension of the homology. If $X$ is a space, and $\cA$ is a cdga model of $X$, then the HTT gives an $A_\infty$-structure on $H^*(\cA)$, from which one obtains a free dgl generated by $H_*(\cA)\cong H_*(X)$. The latter is the Quillen model of $X$.

If one starts with the model of $\Th(\xi)$ obtained in the Theorem \ref{mainthmThom} and applies the HTT, the resulting $A_\infty$-structure is easily described in terms of trees, which can be used to describe the differential in the Quillen model of $\Th(\xi)$. What we have not been able to do, is to give the Quillen model of $\Th(\xi)$ \emph{in terms of the Quillen model of $B$}. We pose it as a question.
\begin{question} Is it possible to describe the Quillen model of $\Th(\xi)$ in terms of the Quillen model of $B$ and the Euler class of $\xi$?
\end{question}

Here are two reasons to look for such a thing:
\begin{enumerate}
\item If one wants to investigate maps to a Thom space, then in the Sullivan approach one needs a cofibrant cdga model of the Thom space, and in the Quillen approach one needs a fibrant dgl model of the Thom space. But our model of Theorem \ref{mainthmThom} is not cofibrant. On the other hand, any dgl is fibrant.
\item  The Quillen model of $\Th(\xi)$ is generated by $H_*(\Th(\xi))$, which is given by the Thom isomorphism theorem. On the other hand, the Sullivan model of $\Th(\xi)$ is generated by the rational homotopy groups of $\Th(\xi)$, which are usually infinite and with exponential growth (for instance, when the vector bundle is trivial).
\end{enumerate}
If the space is formal, then a trivial application of the HTT (or even, as a direct consequence of our main theorem) exhibits the following Quillen model of $\Th(\xi)$:

Let $B$ be a formal space with Quillen model $(\bL(V),d)$. This is equivalent to ask $V$ to have a ordered set of generators $\{v_i\}$ and the differential $d$ has to be quadratic and order-preserving, i.e., $d(v_k) = \sum_{i,j<k}{\lambda_{i,j}[v_i,v_j]}$. Let $\xi\colon E\to B$ be a vector bundle with Euler class $e$. Let $\varphi_e\colon V\to V$ be the dual of the map $V^*\cong H^*(B)\to H^{*+n}(V)\cong V^*$ given by mutiplication by the Euler class. From Theorem \ref{mainthmThom}, one can deduce the following
\begin{corollary} If $B$ is formal and $\xi$ is orientable, the Quillen model of $\Th(\xi)$ is $(\bL(s^nV),\bar{d})$ with differential 
\[\bar{d}(s^nv_k) = \sum_{i,j<k}{\lambda_{i,j}\left([s^n \varphi_e(v_i),s^n v_j] + [s^n v_i,s^n \varphi_e(v_j)]\right)}.\]
\end{corollary}

\section{Weight decompositions on Thom spaces}\label{S3}
In this section, we study (positive) weight decompositions of Thom spaces. We then apply our study
to the theory of representability of cohomology classes by submanifolds.

\subsection{Derationalisation of maps via weight decompositions}

Let $X$ and $Y$ be finite complexes. Assume that we have a map $f\colon X\to Y_\bQ$ from $X$ to the rationalisation of $Y$.
In \cite[Prop.\ 3.1]{Papadima} (see also \cite{Shiga}) it is shown that, whenever $Y$ is a formal simply connected space, the map $f$ admits a \textit{derationalisation}:
there exists a self homotopy equivalence $h\colon Y_\bQ\to Y_\bQ$ such that the map induced in cohomology by $h$ is given by 
multiplication by certain scalars, and the lifting problem
\[\xymatrix{
& & Y\ar[d] \\
X\ar[r]\ar@{-->}[urr] & Y_\bQ\ar[r]^h & Y_\bQ
}\]
has a solution.
In this section, we show that there is a more 
general situation in which the derationalisation problem has a solution:
the existence of positive weight decompositions considered in \cite{Body-Douglas} and \cite{Body-Mimura-Shiga-Sullivan}.

\begin{df}\label{defweightdec}A \emph{weight decomposition} of a cdga $\cA$ is a direct sum decomposition for each $n\geq 0$,
\[\cA^n = \bigoplus_{p\in\mathbb{Z}} \cA_p^n\]
 with $d(\cA_p^n)\subset \cA_{p}^{n+1}$ and $\cA_p^n\cdot \cA_q^m\subset \cA_{p+q}^{n+m}$. 
 A weight decomposition is \emph{positive} if
 $$\cA^n = \bigoplus_{p>0} \cA_p^n\text{ for all }n>0\text{ and }\cA^0=\cA^0_0.$$
  We will say that a space $X$ \textit{admits a (positive) weight decomposition} if there is a cdga model $\cA\longrightarrow \cA_{PL}(X)$ with a (positive) weight decomposition. 
\end{df}

Note that every space admits a trivial weight decomposition of weight 0, which is clearly non-positive. 
However, there are examples of topological spaces which do not admit positive weight decompositions (see \cite{Body-Douglas}).

A weight decomposition on $\cA$ makes its cohomology into a bigraded algebra:
 $$H^n(\cA)_p:=\Ker(d:\cA^n_p\longrightarrow \cA^{n+1}_{p})/\Img(d:\cA_{p}^{n-1}\longrightarrow \cA^n_p).$$

We will use the following Lemma.
\begin{lemma}\label{minimaldec}
Let $\cA$ be a connected cdga with a (positive) weight decomposition. Then a minimal model $\cM\lra \cA$
admits a (positive) weight decomposition.
\end{lemma}
\begin{proof}
It suffices to revise the construction of minimal models of connected cdga's in the bigraded setting.
Let $\cM[0]=\bQ$ be concentrated in degree 0 and weight 0.
We may assume that we have constructed, for all $i<k$, a cdga $\cM[i]$ admitting a positive weight decomposition,
together with a morphism of cdga's $f[i]:\cM[i]\to \cA$ compatible with the weight decompositions and such that $H^k(f[i])$
is an isomorphism for all $k\leq i$ and $H^{i+1}(f[i])$ is a monomorphism.
Denote by $f[i]_p:\cM[i]_p\to \cA_p$ the morphism of complexes given by the restriction to the weight-$p$ complexes.
The cdga $\cM[n]$ is classically defined in two steps. We just need to take care of the weights during these steps. 
First, let $V[n,0]_p:=H^n(C(f_{n-1,p}))$, where $C(-)$ denotes the mapping cone.
Note that, since the weight decompositions of $\cM[n-1]$ and $\cA$ are positive, we have $V[n,0]_p=0$ for all $p\leq 0$.
Consider $V[n,0]_p$ as a bigraded vector space of cohomological degree $n$ and weight $p$. We then let 
$$\cM[n,0]:=\cM[n-1]\otimes_d\Lambda(\bigoplus_{p} V[n,0]_p).$$
A differential $d:V[n,0]_p\to \cM[n-1]_p^{n+1}$ and a morphism 
$f[n,0]_p:V[n,0]_p\to \cA_p^n$ are defined, as in the classical case, by taking a section of the projection
$$Z^n(C(f_{n-1,p}))\twoheadrightarrow H^n(C(f_{n-1,p})).$$
A minor technical detail is that here, $\Lambda(-)$ denotes the free bigraded algebra on a bigraded vector space,
and the multiplicative extension $\otimes$ is done by adding cohomological degrees and weights accordingly.
The $f[n,0]:\cM[n,0]\to \cA$ is such that $H^{k}(f[n,0])$ is an isomorphism for all $i\leq n$.
The second step is done similarly, by inductively killing the kernel of $H^{n+1}(f[n,0])$
through extensions
$$\cM[n,1]:=\cM[n,0]\otimes_d\Lambda(\bigoplus_{p} V[n,1]_p),\text{ where }V[n,1]_p:=H^{n}(C(f[n,0]_p)$$
and iterating this process until $\mathrm{Ker} H^{n+1}(f[n])=0$, where
$f[n]:=\cup_i f[n,i]$.
\end{proof}

Formal and coformal spaces admit positive weight decompositions. In fact, formal spaces are characterised by
having a positive weight decomposition with $H^n(\cA)_n = H^n(\cA)$,
and a minimal model $\Lambda V$ of a coformal space is characterised by having a positive weight
decomposition of the form $V^n_{n-1} = V^n$.

The following is a simple proof of the fact that weight decompositions detect formality.
\begin{proposition}\label{WeightsFormal}
 Let $\cA$ be a cdga with a weight decomposition $\cA^n=\bigoplus \cA^n_p$.
If $H^n(\cA)_p=0$ for all $p\neq n$ then $\cA$ is a formal cdga.
\end{proposition}
\begin{proof}
Let $\tau \cA$ be the bigraded algebra given by the canonical truncation at weight $n$:
$\tau \cA_p^n=0$ for $p<n$, $\tau \cA^n_p={\cA}^n_p$ for $p>n$ and $\tau \cA^n_n=\Ker(d)\cap {\cA}^n_n$ with the differential induced by that of 
${\cA}$. It is straightforward to verify that $\tau \cA,$ is a sub-cdga of $\cA$.
Furthermore, since $H^n({\cA})_p=0$ for all $p\neq n$, the inclusion is clearly a quasi-isomorphism.
Consider the natural projection $\tau \cA\longrightarrow H({\cA})$
defined by $\tau \cA^n_n=\Ker(d)\cap {\cA}^n_n\twoheadrightarrow H^n(\cA)_n$ and sending 
$\tau \cA^n_p$ to $0$, for all $p\neq n$. 
Again, it is clear that this projection is a quasi-isomorphism of cdga's.
\end{proof}

The existence of positive weight decompositions turns out to be sufficient
to derationalise maps:

\begin{theorem}\label{derationalise} Let  $X$ and $Y$ finite type complexes with $X$ of finite dimension and $Y$ simply connected. Assume that $Y$ admits a positive weight decomposition. 
If $f\colon X\to Y_\bQ$ is a map, then there exists a self homotopy equivalence $h\colon Y_\bQ\to Y_\bQ$ such that the map induced in 
cohomology by $h$ is given by multiplication by certain scalars, and the lifting problem
\[\xymatrix{
& & Y\ar[d] \\
X\ar[r]\ar@{-->}[urr] & Y_\bQ\ar[r]^h & Y_\bQ
}\]
has a solution.
\end{theorem}
\begin{proof}
The proof is parallel to that of Proposition 3.1 of \cite{Papadima}, using the characterisation
of weight decompositions of minimal cdga's of \cite{Body-Mimura-Shiga-Sullivan}.
We sketch the main steps.
By Lemma $\ref{minimaldec}$ we can assume that $Y$ has a minimal model with a positive weight decomposition.
By \cite[Prop.\ 2.3]{Body-Mimura-Shiga-Sullivan},
 for each $\lambda\in \bZ\setminus \{0\}$, there is a self-homotopy equivalence of $Y_\bQ$ such that the 
 induced homomorphism on homotopy groups is given by multiplication by a power of $\lambda$.
 On the other hand, the homotopy of the homotopy fibre of $Y\to Y_{\bQ}$ consists only on 
 finitely many torsion elements in each dimension. Take $\lambda$ to be a multiple of the
 order of all the torsion elements of the homotopy groups in dimension at most $\dim X - 1$.
 Then, the pushforward along $h_\lambda$ of the obstructions to the existence of a lift vanishes, and therefore there is a lift.
\end{proof}

\subsection{Formality and weight decompositions of Thom spaces}
We next use the model of the Thom space given in Theorem $\ref{mainthmThom}$ to relate
weight decompositions on the base with the existence of weight decompositions on the Thom space.

We first study the formality of Thom spaces.
While
Thom spaces of universal bundles of classifying spaces of connected closed subgroups of $GL(\bR^n)$ are formal (see \cite{Papadima}),
in general, this is not the case. We give an example of a non-formal 
Thom space in Example $\ref{exnonformal}$ using the non-vanishing of Massey products.
As a straightforward consequence of Theorem $\ref{mainthmThom}$, we show that formality of the base is transferred to formality of the Thom space.

\begin{proposition} Let $\xi\colon E\to B$ be a vector bundle. If $B$ is formal then the Thom space $\Th(\xi)$ is formal.
\end{proposition}
\begin{proof} If $B$ is formal, then we take $\cA = H^*(B)$ with trivial differential as a cdga model of $B$, 
and the theorem gives that $(s^n H^*(B),0)$ with the multiplication given by 
the Thom isomorphism theorem, is a model of $\Th(\xi)$, hence $\Th(\xi)$ is formal.
\end{proof}
 
\begin{remark} The Thom space of a vector bundle of rank $n$ is $(n-1)$-connected and, therefore, if the dimension of the base is at most $2n-2$ (so that the dimension of the Thom space is at most $3n-2$), one can use \cite[Corollary~5.16]{Halperin-Stasheff} to deduce that $\Th(\xi)$ is intrinsically formal. 
For instance, this covers the case of tangent bundles of manifolds.
A modification of this argument was used in the first author's thesis to prove that the Thom space
$\Th(\gamma_{2,n}^\perp)$ of the complement of the tautological bundle over the oriented Grassmannian
of $2$-planes in $\bR^{n+2}$ is intrinsically formal. The $n$-fold loop space of $\Th(\gamma_{2,n}^\perp)$ 
appeared in \cite{CR-W} as the target of the scanning map of the moduli space of surfaces in $\bR^n$, and 
the motivation to write this paper came at first from the need of a description of the rational homotopy
type of the target of the scanning map for moduli spaces of high dimensional submanifolds of $\bR^n$. 
This target is the $n$-fold loop space of a certain Thom space $\Th(\theta^*\gamma_{d,n}^\perp)$ similar 
to $\Th(\gamma_{d,n}^\perp)$. As long as $d>2$, the argument of Halperin and Stasheff does not apply to these Thom spaces. 
\end{remark}

As for the transfer of weight decompositions, we have:

\begin{proposition}\label{cincopuntotres} Let $\xi\colon E\to B$ be a vector bundle. Assume that $B$ admits a (positive) weight decomposition
such that the Euler class is homogeneous in the decomposition.
Then $\Th(\xi)$ admits a (positive) weight decomposition.
\end{proposition}
\begin{proof}
Define the weight $||w_x||$ of $w_x$ as the weight of $x$ plus the weight of the Euler class: 
$||w_x||=||x||+||e||$.  This gives a well-defined weight decomposition on $\Mm$, since 
\[||dw_x||=||w_{dx}|| = ||dx|| + ||e||=||x|| + ||e|| = ||w_x||\]
 and
\[
||w_x\cdot w_y|| = ||w_{exy}|| = ||exy|| + ||e||
= ||ex||+||ey|| = ||w_x||+||w_y||.\qedhere\]
\end{proof}

\subsection{Homology classes representable by submanifolds}
Given an oriented closed manifold $M$, a cohomology class $c\in H^q(M;\bZ)$ and an oriented submanifold $W\subset M$ of codimension $q$, we say that \textit{$W$ represents $c$} if the fundamental class of $W$ is Poincaré dual to $c$. By work of Thom \cite{Thom:quelques}, this happens if and only if $c$ is the pullback of the Thom class under the Pontryagin-Thom collapse map associated to $W\subset M$:
\[M\lra \Th(\nu W) \lra\Th(\gamma_{q}).\]
Here $\nu W$ denotes the normal bundle of $W$ in $M$ and $\gamma_{q}$ is the universal bundle over $B\SO(q)$.
More generally, it is natural to ask under which conditions a cohomology class $c\in H^q(M)$ is representable by a submanifold. 
This is equivalent to asking the conditions for the existence of a map $\phi\colon M\to \Th(\gamma_{q})$ with $\phi^*(u)=c$. This question was answered by Thom himself.

A \emph{structure} is a fibration $\theta\colon X\to BSO(q)$. A \emph{$\theta$-structure} on a vector bundle $\xi\colon X\to B$ is a lift of the classifying map of $\xi$ along the fibration $\theta$. This factorisation gives a pair of pullback diagrams
\[\xymatrix{
E\ar[d]\ar[r] & \theta^*\gamma_{q} \ar[d]\ar[r] & \gamma_{q}\ar[d] \\
B\ar[r] & X \ar[r] & \Gr_q(\bR^\infty) \rlap{ $= B\SO(q)$}
}\]
and therefore a factorisation
\[\Th(\xi)\lra \Th(\theta^*\gamma_{q})\lra \Th(\gamma_{q}).\]

Examples of $\theta$-structures are framings (when $B =  V_{k,n}$ is the Stiefel manifold of $k$-frames in $\bR^\infty$ and $\theta$ is the map that sends a frame to the oriented plane that it spans), spin structures (when $B$ is the $2$-connected cover of $BSO(q)$) or complex structures (when $q=2k$ and $B = BU(k)$).

We have that a cohomology class $c$ can be represented by a submanifold $W$ with a $\theta$-structure
in its normal bundle if and only if there is a map $\phi\colon M\to \Th(\theta^*\gamma_{d,\infty})$ such that $\phi^*(u)=c$,
where $u$ denotes the Thom class.

Let $M$ be an oriented closed manifold and $\theta\colon B\to BSO(n)$ a structure. Let $\xi$ be
the vector bundle classified by $\theta$. 
Recall from the introduction, that a cohomology class $c\in H^n(M)$ can be represented by a $\theta$-submanifold of
codimension $n$ if and only if the following lifting problem has a solution:
 \[\xymatrix{
& \Th(\xi)\ar[d]^{u} \\
M\ar@{-->}[ur]\ar[r]^-{c} & K(\bZ,n),
}\]
where the vertical arrow is the Thom class. 

The following Theorem asserts that, when $B$ has a positive weight decomposition,
it is enough to solve the above lifting problem at the rational level.

\begin{theorem}\label{representability} Let $\theta\colon B\to BSO(n)$ be a structure and let $\xi\colon E\to B$ be the resulting vector bundle. Suppose that $B$ admits a positive weight decomposition and that the Euler class is homogeneous in the decomposition. Let $M$ be an oriented closed submanifold, and let $c\in H^n(M)$ be a cohomology class. Then there exists a multiple $\lambda c$ with $\lambda\in \bZ\setminus \{0\}$ that is representable by a $\theta$-submanifold if and only if the lifting problem
 \begin{equation*}\xymatrix{
& \Th(\xi)_\bQ\ar[d]^{u} \\
M\ar@{-->}[ur]\ar[r]^-{c} & K(\bZ,n).
}\end{equation*}
has a solution.
\end{theorem}
\begin{proof}
By Proposition $\ref{cincopuntotres}$, the positive weight decomposition on $B$ induces a positive weight decomposition on the Thom space $\Th(\xi)$.
Therefore by Theorem $\ref{derationalise}$, 
the lifting problem
$$
\begin{gathered}
\xymatrix{
& \Th(\xi)\ar[d] \\
M\ar@{-->}[ur]\ar[r] & \Th(\xi)_\bQ.
}\end{gathered}
$$
has a solution.
\end{proof}

\begin{corollary}

\end{corollary}

\subsection{Massey products of Thom spaces}
As is well-known, formality is closely linked to the vanishing of higher Massey products in cohomology.
We end this section by studying Massey products of Thom spaces and the effect of the canonical inclusion on them.

\begin{lemma}\label{lemma:massey3} The triple Massey product $(w_x,w_y,w_z)$ in $H^*(\Th(\xi))$ is in canonical bijection with $e\cdot (x,ey,z) = (ex,y,ez)\in H^*(B)$, and this bijection preserves $0$.
\end{lemma}
\begin{proof}
The Massey product is
\begin{align*}
(w_x,w_y,w_z) &= \{[w_x w_\mu  - w_\lambda  w_z]\mid d(w_\lambda) = w_xw_y\text{ and }d(w_\mu)=w_yw_z\} \\
 &= \{[w_{ex\mu-e\lambda z}]\mid d(\lambda) = exy \text{ and }d(\mu) = eyz\} \\
 &= \{[w_{e(x\mu-\lambda z)}]\mid d(\lambda) = x(ey) \text{ and }d(\mu) = (ey)z\} \\
 &= \{[w_{eu}]\mid u\in (x,ey,z)\} 
\end{align*}
The latter subset of $H^*(\Th(\xi))$ is in bijection with $e\cdot (x,ey,z)\subset H^*(B)$ because $d(w_{eu}) = w_{d(eu)}$. Moreover, this bijection preserves $0$. The case of $(ex,y,ez)$ is done similarly.
\end{proof}
As a direct consequence, we have that not every Thom space is formal. 

\begin{example}\label{exnonformal}
Let $X$ be the cofibre of a non-trivial triple Whitehead product $S^4\to S^2\vee S^2$, and let $\xi\colon E\to X\times \CP^\infty$ be a vector bundle with Euler class $e$ the pullback of the generator of $H^2(\CP^\infty)$. Let $x,y$ be the pullback of the natural generators of $H^2(S^2\vee S^2)$. Then, writing down the minimal model of $X\times \CP^\infty$, one rapidly realises that $(x,ex,y)$ is a non-empty Massey product that does not contain zero. Since multiplication by $e$ is injective, the set $e(x,ex,y)$ is also non-empty and does not contain zero, and so does $(w_x,w_x,w_y)$, therefore $\Th(\xi)$ is not formal.
\end{example}

\begin{remark}
In \cite{Rudyak-Tralle}, a method to produce non-formal symplectic manifolds using Massey products in Thom spaces was introduced. Their argument, which assumes
a pre-existing non-formal symplectic manifold, is the following:
\begin{enumerate}
\item Produce a smooth embedding $i\colon M\to X$ between symplectic manifolds such that $M$ has a non-trivial triple Massey product.
\item Obtain the blow up $\widetilde{X}$ of $X$ along $i$, which is a symplectic manifold that comes with a smooth embedding $\tilde{i}\colon \bP(\nu M)\to \widetilde{X}$ of codimension $2$.
\item Prove that if $M$ has a non-trivial triple Massey product $(x,y,z)$, then $\bP(\nu M)$ has a non-trivial triple Massey product $(ex,ey,ez)$ as long as the complex codimension of $M$ in $X$ is at least $4$.
\item Prove that under the composition $\bP(\nu M)\to \widetilde{X} \to \Th(\nu(\bP(\nu M)))$, the triple Massey product $(w_x,w_y,w_z)$ exists and is mapped to $(ex,ey,ez)$. 
\item Therefore, $(w_x,w_y,w_z)$ is also non-trivial, and there is an intermediate non-trivial triple Massey product in $\widetilde{X}$, so $\widetilde{X}$ is a new non-formal symplectic manifold.
\end{enumerate}
\end{remark}
For higher Massey products there are no nice formulas, but we get maps: First, there is a map
\[e^{k-1}(a_{ii})_{i=1}^k\lra (w_{a_{ii}})_{i=1}^k,\]
but this map can be improved, as it factors through the map in the next lemma. 
\begin{lemma}\label{lemma:masseysup} Let $\varphi(i)$ be $0$ if $i$ is even and $1$ if $i$ is odd, and let $\gamma(i)=\varphi(i)-1$. There are injective pointed maps between Massey products in $B$ and Massey products in $\Th(\xi)$:
\begin{align*}
e^{\lfloor\frac{k-2}{2}\rfloor}\left(e^{\varphi(i)}x_{ii}\right)_{i=1}^k&\lra \left(x_{ii}\cup u\right)_{i=1}^k\\
e^{\lceil\frac{k-2}{2}\rceil}\left(e^{\gamma(i)}x_{ii}\right)_{i=1}^k&\lra \left(x_{ii}\cup u\right)_{i=1}^k.
\end{align*}
when $k=3$, these are the bijections of Lemma \ref{lemma:massey3}. The composition of this map and the pullback along the canonical inclusion is given by multiplication by $e$.
\end{lemma}
\begin{proof}

Let $(A,d)$ be a cdga, and let $\{x_{i,i}\}_{i=1}^k\subset A$. A \emph{Massey system} for $\{x_{i,i}\}_{i=1}^k$ is a pair of sequences 
\[\{a_{i,j}\}_{1\leq i\leq j\leq k},\quad \{m_{i,j}\}_{1\leq i\leq j\leq k, (i,j)\neq (1,k)}\]
in $A$, satisfying that
\begin{enumerate}
\item $a_{i,i} = 0$,
\item $m_{i,i} = x_{i,i}$, 
\item $d(m_{i,j}) = a_{i,j}$,
\item $a_{i,j} = \sum_{l=i}^{j-1} \pm m_{i,l}m_{l+1,j}$.
\end{enumerate}
The product associated to this system is $a_{1,k}$. Finally, the Massey product of $(x_{i,i})_{i=1}^k$ is the set of all products of Massey systems for $(x_{i,i})_{i=1}^k$. 

Define, depending on whether we are constructing the first or the second map,
\[s(i,j) = \lfloor \frac{j-i-\varphi(i)}{2}\rfloor, \quad s(i,j) = \lfloor \frac{j-i-\gamma(i)}{2}\rfloor\]
and observe that 
\begin{enumerate}
\item\label{prop:11} $s(i,j) = s(i,l) + s(l+1,j)+1$,
\item $s(1,k) = \lfloor\frac{k-2}{2}\rfloor$ for the first map and $s(1,k) = \lceil\frac{k-2}{2}\rceil$ for the second map.
\end{enumerate}
Now, let $(A,d)$ be a cdga model of $B$ and let $(T,\delta)$ be the model of $\Th(\xi)$ given by Theorem \ref{mainthmThom}. Let $y_{i,i} = e^{\varphi(i)}x_{i,i}$ (resp.\ $y_{i,i} = e^{\gamma(i)}x_{i,i}$), and let $\{a_{i,j}\}$, $\{m_{i,j}\}$ be a Massey system for $\{y_{i,i}\}_{i=1}^k\subset A$. We claim that 
\begin{equation}\label{eq:4}
\left\{\left(e^{s(i,j)}a_{i,j}\right)\cup \bar{u}\right\}, \left\{\left(e^{s(i,j)}m_{i,j}\right)\cup \bar{u}\right\}
\end{equation}
is a Massey system for $\{x_{i,i}\cup \bar{u}\}_{i=1}^k\subset T$. The first three conditions are straightforward to check, and for the fourth one uses the additivity property \eqref{prop:11} of $s(i,j)$ remarked before:
\begin{align*}
\left(e^{s(i,j)}a_{i,j}\right)\cup \bar{u} &= \left(e^{s(i,j)}\sum_{l=i}^{j-1} \pm m_{i,l}m_{l+1,j}\right)\cup \bar{u} \\
&= \left(\sum_{l=i}^{j-1} \pm e\cdot\left(e^{s(i,l)}m_{i,l}\right)\cdot \left(e^{s(l+1,j)}m_{l+1,j}\right)\right)\cup \bar{u} \\
&= \sum_{l=i}^{j-1} \pm \left(\left(e^{s(i,l)}m_{i,l}\right)\cup \bar{u}\right)\cdot\left(\left(e^{s(l+1,j)}m_{l+1,j}\right)\cup \bar{u}\right)
\end{align*}

This defines a map from Massey systems of $\{y_{i,i}\}_{i=1}^k$ to Massey systems of $\{x_{i,i}\cup \bar{u}\}_{i=1}^k$, which after taking products of the Massey systems, is given by multiplying the product $a_{1,k}$ by the adequate power of the Euler class, and then taking the map $-\cup \bar{u}$ inducing the Thom isomorphism. That these maps send zero to zero is clear from the construction: if $a_{1,k}$ is a coboundary, then $\left(e^{\lfloor \frac{k-1}{2}\rfloor}a_{1,k}\right)\cup u$ is a coboundary too.

Now, we look at the pullback of \eqref{eq:4} along the canonical map $\iota\colon B\to \Th(\xi)$, which is
\[\{e^{s(i,j)+1}a_{i,j}\}, \{e^{s(i,j)+1}m_{i,j}\},\]
and so the composite of the map of the lemma and the pullback is multiplication by the Euler class.
\end{proof}

\section{Positive weight decompositions for smooth varieties and application to motivic Thom spaces}\label{S4}
In this section, we show
that mixed Hodge theory naturally gives rise to functorial weight decompositions, which are always positive for smooth varieties.
We then study Thom spaces of complex vector bundles on smooth complex varieties.
Specifically, we use Theorem \ref{mainthmThom} to describe the mixed Hodge structures on the rational homotopy type of motivic Thom spaces.

\subsection{Mixed Hodge structures give weight decompositions}
We first study the relation of mixed Hodge structures with weight decompositions.
\begin{df}
A \textit{mixed Hodge structure} on a rational vector space $V$ is given by an increasing filtration $W$ of $V$,
called the \textit{weight filtration},
together with a decreasing filtration $F$ on $V_\CC:=V\otimes_{\QQ}\CC$, called the 
\textit{Hodge filtration}, such that 
for all $m\geq 0$, each graded vector space $Gr_m^WV:=W_mV/W_{m-1}V$ carries a pure Hodge structure of weight $m$
given by the filtration induced by $F$ on 
$Gr_m^WV\otimes\CC$, i.e., there is a direct sum decomposition
$$Gr_m^WV\otimes\CC=\bigoplus_{p+q=m}V^{p,q}\text{ where }V^{p,q}=F^p(Gr_m^WV\otimes\CC)\cap \overline{F}^q(Gr_m^WV\otimes\CC)=\overline{V}^{q,p}.$$
\end{df}
Morphisms of mixed Hodge structures are given by morphisms $f:V\to V'$ of vector spaces compatible with filtrations:
$f(W_mV)\subset W_mV'$ and $f(F^pV_\CC)\subset F^pV_\CC'$.

To compare mixed Hodge structures and weight decomposition we will use the following result on the splitting of mixed Hodge structures:
\begin{lemma}[\cite{DeHII}, 1.2.11]\label{Desplitting}
Let $(V,W,F)$ be a mixed Hodge structure. Then $V_\CC:=V\otimes_\QQ\CC$ admits
a direct sum decomposition
$V_\CC=\bigoplus I^{i,j}$ such that the filtrations $W$ and $F$ defined on $V_\CC$ are given by
$$W_mV_\CC=\bigoplus_{i+j\leq m} I^{i,j}\text{ and }F^pV_\CC=\bigoplus_{i\geq p} I^{i,j}.$$
This decomposition is functorial for morphisms of mixed Hodge structures.
\end{lemma}

\begin{df}
A \textit{filtered cdga} $(\cA,W)$ is a cdga $\cA$ together with a filtration
$\{W_p\cA^n\}$ indexed by $\ZZ$ on each $\cA^n$ such that $W_{p-1}\cA^n\subset W_p\cA^n$,
$d(W_p\cA^n)\subset W_pA^{n+1}$, and $W_p\cA^n\cdot W_q\cA^m\subset W_{p+q}\cA^{n+m}$.
\end{df}

\begin{remark}
Let $\cA$ be a cdga. Then a weight decomposition $\cA^n=\bigoplus \cA_p^n$ 
for $\cA$ gives a filtered cdga $(\cA,W)$
 by letting
$W_p\cA^n:=\bigoplus_{q\leq p}\cA^n_q$. 
The converse is not true in general: if $(\cA,W)$ is a filtered cdga, 
one cannot always find a weight decomposition $\cA^n=\bigoplus \cA_p^n$ satisfying 
$W_p\cA^n:=\bigoplus_{q\leq p}\cA^n_q$. 
We will next see that mixed Hodge structures ensure this is possible,
\end{remark}

\begin{df}\label{df:mhs}
A \textit{mixed Hodge cdga} is a filtered cdga $(\cA,W)$ defined over $\QQ$,
together with a filtration $F$ on $\cA_\CC:=\cA\otimes\CC$, such that 
\begin{enumerate}[(i)]
\item\label{it:mhs1} for each $n\geq 0$, the triple $(\cA^n,W,F)$ is a mixed Hodge structure and
\item\label{it:mhs2} the differentials $d:\cA^n\to \cA^{n+1}$ and products $\cA^n\times \cA^m\to \cA^{n+m}$ are morphisms of mixed Hodge structures.
\end{enumerate}
\end{df}

\begin{lemma}\label{MHSWeightdec}
Every mixed Hodge cdga $\cA$ admits a
weight decomposition 
$\cA^n=\bigoplus \cA_p^n$ 
such that $W_p\cA^n=\bigoplus_{q\leq p}\cA^n_q$.
\end{lemma}
\begin{proof}
We will follow a similar argument to that of Lemma 3.20 of \cite{CG1}
on the relation between a mixed Hodge cdga and its associated weight spectral sequence (see also \cite{Mo}, Theorem 9.6). 
By Lemma $\ref{Desplitting}$ we have functorial decompositions
$$\cA^n_\CC=\bigoplus I^{p,q}_n,\text{ with }W_p \cA^n_\CC=\bigoplus_{i+j\leq p} I^{i,j}_n.$$
Since both the differential and products of $\cA$ are morphisms of mixed Hodge structures, we have
$$d(I^{p,q}_n)\subset I^{p,q}_{n+1}\text{ and }I^{p,q}_n\cdot I^{p',q'}_{n'}\subset I^{p+p',q+q'}_{n+n'}.$$
Define 
$$\cA_p^n:=\bigoplus_{i+j=p} I_n^{i,j}.$$
We next check that this is a weight decomposition for $\cA_\CC$. Indeed, we have:
$$d(\cA_p^n)=\bigoplus_{i+j=p} dI_n^{i,j}\subset \bigoplus_{i+j=p} I_{n+1}^{i,j}=\cA_{p}^{n+1}.$$
As for the compatibility of products:
$$\cA_p^n\cdot \cA_{p'}^{n'}=\bigoplus_{i+j=p} I_{n}^{i,j}\cdot \bigoplus_{i'+j'=p'} I_{n'}^{i',j'}
\subset \bigoplus\limits_{\substack{i+j=p\\i'+j'=p'}}I_{n+n'}^{i+i',j+j'}=\cA_{p+p'}^{n+n'}.$$
Clearly, we have that $W_p\cA^n_\CC=\bigoplus_{q\leq p}\cA^n_q$.
Now, by the theory of descent of splittings of \cite[Theorem 4.3]{CH},
this weight decomposition descends to a decomposition over $\bQ$.
\end{proof}

\subsection{Mixed Hodge structures in rational homotopy}
Deligne's weight filtration on the rational cohomology 
$H^k(X;\QQ)$ of a complex algebraic variety $X$
is bounded by: 
$$0=W_{-1}H^k(X;\QQ)\subset W_0H^k(X;\QQ)\subset\cdots\subset W_{2k}H^{k}(X;\QQ)=H^k(X;\QQ).$$
If $X$ is projective then $W_{k}H^{k}(X;\QQ)=H^k(X;\QQ)$, while if $X$ is smooth then
$W_{k-1}H^{k}(X;\QQ)=0$. In particular, for a smooth projective variety,
the weight filtration on $H^k(X;\QQ)$ is pure of weight $k$.

It was proven by Morgan \cite{Mo} for smooth complex varieties and by Hain \cite{Ha} and Navarro-Aznar \cite{Na}
independently for possibly singular varieties, that the rational homotopy type of
every complex algebraic variety carries mixed Hodge structures in a certain functorial way.
A consequence of this result is the following:
\begin{theorem}[\cite{CG1}, Theorem 3.17]\label{existence}
For every complex algebraic variety $X$ there exists a
mixed Hodge cdga $(\cA,W,F)$ such that $\cA\simeq \cA_{PL}(X)$,
with
$$0=W_{-1}\cA^k\subset W_0\cA^k\subset \cdots\subset W_{p-1}\cA^k\subset W_p\cA^k\subset\cdots\subset W_{2k}\cA^k=\cA^k.$$
Furthermore, if $X$ is projective then $W_{k}\cA^k=\cA^k$, while if $X$ is smooth then $W_{k-1}\cA^k=0$.
The induced mixed Hodge structure on $H^k(X;\QQ)$ is Deligne's mixed Hodge structure. 
\end{theorem}
\begin{remark}
For the smooth case, the above result is due to Morgan \cite{Mo}.
The definition of mixed Hodge cdga of \cite{CG1} differs by a d\'{e}calage from the one introduced here.
This does not affect the above result. Also, in Theorem 3.17 of \cite{CG1},
the bounds on the weight filtration are not stated explicitly.
However, they can be directly deduced from the proof of the theorem.
\end{remark}

\begin{theorem}\label{AlgVarsWeightDec}
Let $X$ be a complex algebraic variety. Then its homotopy type 
admits a functorial
weight decomposition 
$\cA=\bigoplus_{p\geq 0}\cA_p^n$ such that $H^n(\cA)_{p}\cong Gr_p^WH^n(X;\bQ)$, where $W$ denotes Deligne's weight filtration. 
If $X$ is a smooth variety then the weight decomposition is positive.
\end{theorem}
\begin{proof}
By Theorem $\ref{existence}$ we may assume that $\cA\simeq  \cA_{PL}(X)$ is a mixed Hodge cdga
such that $W_{-1}\cA^n=0$ ($W_{n-1}\cA^n=0$ if $X$ is smooth).
The proof now follows from Lemma $\ref{MHSWeightdec}$.
\end{proof}

\subsection{Motivic Thom spaces}
Let $\xi:E\longrightarrow B$ be a rank $k$ complex vector bundle over a smooth complex variety $B$.
Denote by $s:B\longrightarrow E$ the zero section. 
The Thom space of $\xi$ is defined by
the homotopy quotient
$$\mathrm{Th}(\xi)=E/(E-s(B)).$$
Here the complement of the zero section is the algebraic analogue of the sphere bundle.
The homotopy equivalences $E\simeq D(E)$ and $E-s(B)\simeq S(B)$ allow us to identify this construction with the
topological definition of the Thom space.

In general, the Thom space of an algebraic vector bundle is not an algebraic variety.
This is due to the nonexistence of colimits in algebraic geometry.
However, it is by construction a motivic space in the sense of Voevodsky \cite{Voevodsky}.
In particular, via the mixed Hodge realisation functor, the cohomology of $\mathrm{Th}(\xi)$
carries a mixed Hodge structure (see \cite{Hu}). In fact, the mixed Hodge structure
may also be obtained using the relative cohomology of $B$ and $E-s(B)$ (see 8.3.8 of \cite{DeHIII}).

Furthermore, being defined as a homotopy limit, the rational homotopy type
of $\mathrm{Th}(\xi)$ carries mixed Hodge structures compatible with the natural map $B\longrightarrow \mathrm{Th}(\xi)$.
Indeed,
a mixed Hodge diagram for $\mathrm{Th}(\xi)$ may be constructed by taking the Thom-Whitney simple
$$s_{\mathrm{TW}}\left(\cA(*)\times \cA(E)\rightrightarrows \cA(E-s(B))\right),$$
where $\cA:\mathrm{Sch}_\CC\longrightarrow \mathrm{Ho}(\mathrm{MHD})$ denotes Navarro-Aznar's functor 
with values on the homotopy category of mixed Hodge diagrams (see \cite{Na} for details).
We next use the model of the Thom space given in Theorem \ref{mainthmThom} to 
provide an explicit simple description of these mixed Hodge structures, 
in terms of the mixed Hodge structures on $\cA_{PL}(B)$.

\begin{theorem}\label{MHSenelmodelo}
Let $\xi:E\longrightarrow B$ be a rank $k$ complex vector bundle over a smooth complex variety $B$.
The rational homotopy type $\cA_{PL}(\mathrm{Th}(\xi))$ of the Thom space carries mixed Hodge
structures compatible with the map $\iota:B\longrightarrow \mathrm{Th}(\xi)$. Furthermore:
\begin{enumerate}[(1)]
 \item The induced mixed Hodge
structure in cohomology coincides with Deligne's mixed Hodge structure. 
\item The morphism  $-\cup u:\cA^*\lra \cA[e]^{*+k}$ of Theorem $\ref{mainthmThom}$ 
becomes an isomorphism of differential graded mixed Hodge structures.
\end{enumerate}
\end{theorem}
\begin{proof}
By Theorem $\ref{existence}$ we can take a model of $B$ to be a mixed Hodge cdga $(\cA,W,F)$.
Let $\Mm$ be the cdga model of $\mathrm{Th}(\xi)$ given by Theorem \ref{mainthmThom}.
Define filtrations $W$ and $F$ on $\Mm$ by letting 
$$W_p\Mm^n:=W_{p-2k}(s^{2k}\cA)^n=W_{p-2k}\cA^{n-2k}\text{ and }F^p\Mm^n_\CC:=F^{p-k}(s^{2k}\cA_\CC)^n.$$
The characteristic classes of a complex fiber bundle over every complex algebraic variety are pure of type $(k,k)$
(see for example \cite{DeHIII}, Corollary 9.1.3).
Since the Euler class is the top Chern class, this implies that ${e}$ is pure of type $(k,k)$. Therefore we have 
${e}\in W_{2k}\cA^{2k}\setminus W_{2k-1}\cA^{2k}$ and ${e}\otimes\CC\in F^k\cA^{2k}\setminus F^{k-1}\cA^{2k}$.
As a consequence, the filtrations $W$ and $F$ defined above for $\Mm$ 
are compatible with products and differentials, hence condition \eqref{it:mhs2} in 
Definition \ref{df:mhs} holds.
It only remains to check condition \eqref{it:mhs1}, i.e., that for all $n\geq 0$, the triple $(\Mm^n,W,F)$ is a mixed Hodge structure.
This follows from the fact that $\Mm^n=\cA^{n-2k}(-k)$ is the $-k$-Tate twist of $\cA^{n-2k}$ 
(see for example Section 3.1 of \cite{PS}). We next provide a quick proof for completeness:
Since $(\cA^{n-2k},W,F)$ is a mixed Hodge structure, we have a direct sum decomposition 
$$Gr_m^W\cA^{n-2k}_\CC=\bigoplus_{p+q=m}V^{p,q}$$ where $V^{p,q}=F^p\cA^{n-2k}_\CC\cap \overline{F}^{q}\cA^{n-2k}_\CC.$
Let us show that the vector spaces $U^{p,q}:=V^{p-k,q-k}$ give the decomposition for $(\cA[e]^n,W,F)$. Indeed, we have:
$$Gr^W_m\cA[e]^n_\CC=Gr^W_{m-2k}\cA^{n-2k}_\CC=\bigoplus_{p+q=m-2k}V^{p,q}=\bigoplus_{p+q=m}U^{p,q}.$$
Therefore we have
$$U^{p,q}=V^{p-k,q-k}=F^{p-k}\cA^{n-2k}_\CC\cap \overline{F}^{q-k}\cA^{n-2k}_\CC=F^p\cA[e]^n_\CC\cap \overline{F}^q\cA[e]^n_\CC.$$
This proves that $(\Mm,W,F)$ is a mixed Hodge cdga.

By construction, the morphism $\Mm\longrightarrow \cA$ defined via multiplication by ${e}$ is compatible with filtrations,
hence a morphism of mixed Hodge cdga's.

We now prove (1). The mixed Hodge structure induced on $H^n(\mathrm{Th}(\xi);\QQ)$ is just the $-k$-Tate twist
of $H^{n-2k}(B)$. 
This is the only mixed Hodge structure making the map
 $H^*(\mathrm{Th}(\xi);\QQ)\cong H^{*-2k}(B)\longrightarrow H^*(B)$,
given by $a\mapsto a\cdot {e}$, strictly compatible with filtrations. 
Therefore it coincides with Deligne's mixed Hodge structure. 
The isomorphism $-\cup u:\cA^*\lra \cA[e]^{*+2k}$ of (2) now follows trivially.
 \end{proof}

\begin{corollary}
Let $\xi:E\longrightarrow B$ be a rank $k$ complex vector bundle over a smooth complex variety $B$.
The rational homotopy type of the Thom space $\mathrm{Th}(\xi)$ 
carries a positive weight decomposition. In particular, Theorem $\ref{representability}$ applies.
\end{corollary}
\begin{proof}
This follows from the proof of the previous theorem together with Lemma $\ref{MHSWeightdec}$. 
Alternatively, by Theorem \ref{AlgVarsWeightDec}, $B$ 
carries a positive weight decomposition.
Note that a key point in the proof of Theorem $\ref{MHSenelmodelo}$ is the fact that the Euler class
is pure of type $(k,k)$ in $H^{2k}(B;\QQ)$. In particular, the
hypotheses of Proposition $\ref{cincopuntotres}$ are satisfied.
\end{proof}

\bibliographystyle{amsalpha}
\bibliography{biblio-article}

\def\cprime{$'$}
\providecommand{\bysame}{\leavevmode\hbox to3em{\hrulefill}\thinspace}
\providecommand{\MR}{\relax\ifhmode\unskip\space\fi MR }
\providecommand{\MRhref}[2]{%
  \href{http://www.ams.org/mathscinet-getitem?mr=#1}{#2}
}
\providecommand{\href}[2]{#2}
\begin{thebibliography}{CMRW16}

\bibitem[BD78]{Body-Douglas}
Richard Body and Roy Douglas, \emph{Rational homotopy and unique
  factorization}, Pacific J. Math. \textbf{75} (1978), no.~2, 331--338.

\bibitem[BMSS98]{Body-Mimura-Shiga-Sullivan}
Richard Body, Mamoru Mimura, Hiroo Shiga, and Dennis Sullivan,
  \emph{{$p$}-universal spaces and rational homotopy types}, Comment. Math.
  Helv. \textbf{73} (1998), no.~3, 427--442. \MR{1633367}

\bibitem[BT82]{BottTu}
Raoul Bott and Loring~W. Tu, \emph{Differential forms in algebraic topology},
  Graduate Texts in Mathematics, vol.~82, Springer-Verlag, New York-Berlin,
  1982.

\bibitem[Bur71]{Burlet}
Oscar Burlet, \emph{Cobordismes de plongements et produits homotopiques},
  Comment. Math. Helv. \textbf{46} (1971), 277--288. \MR{0295367}

\bibitem[CG14]{CG1}
Joana Cirici and Francisco Guill{\'e}n, \emph{{$E\sb 1$}-formality of complex
  algebraic varieties}, Algebr. Geom. Topol. \textbf{14} (2014), no.~5,
  3049--3079.

\bibitem[CH17]{CH}
Joana Cirici and Geoffroy Horel, \emph{Mixed {H}odge structures and formality
  of symmetric monoidal functors}, arXiv:1 math/703.06816, 2017.

\bibitem[CMRW16]{CR-W}
Federico Cantero~Mor\'an and Oscar Randal-Williams, \emph{{Homological
  stability for spaces of surfaces}}, arXiv: 1304.3006v3, to appear in Geometry
  \& Topology (2016).

\bibitem[Del71]{DeHII}
Pierre Deligne, \emph{Th\'eorie de {H}odge. {II}}, Inst. Hautes \'Etudes Sci.
  Publ. Math. (1971), no.~40, 5--57.

\bibitem[Del74]{DeHIII}
\bysame, \emph{Th\'eorie de {H}odge. {III}}, Inst. Hautes \'Etudes Sci. Publ.
  Math. (1974), no.~44, 5--77.

\bibitem[FHT01]{FHT}
Yves F\'elix, Stephen Halperin, and Jean-Claude Thomas, \emph{Rational homotopy
  theory}, Graduate Texts in Mathematics, vol. 205, Springer-Verlag, New York,
  2001.

\bibitem[FOT16]{FOT16}
Yves F{\'e}lix, John Oprea, and Daniel Tanr{\'e}, \emph{Lie-model for {T}hom
  spaces of tangent bundles}, Proc. Amer. Math. Soc. \textbf{144} (2016),
  no.~4, 1829--1840. \MR{3451257}

\bibitem[Hai87]{Ha}
Richard Hain, \emph{The de {R}ham homotopy theory of complex algebraic
  varieties. {II}}, $K$-Theory \textbf{1} (1987), no.~5, 481--497.

\bibitem[HS79]{Halperin-Stasheff}
Stephen Halperin and James Stasheff, \emph{Obstructions to homotopy
  equivalences}, Adv. Math. \textbf{32} (1979), no.~3, 233--279.

\bibitem[Hub95]{Hu}
Annette Huber, \emph{Mixed motives and their realization in derived
  categories}, Lecture Notes in Mathematics, vol. 1604, Springer-Verlag,
  Berlin, 1995.

\bibitem[LV12]{Loday-Vallette}
Jean-Louis Loday and Bruno Vallette, \emph{Algebraic operads}, Grundlehren der
  Mathematischen Wissenschaften [Fundamental Principles of Mathematical
  Sciences], vol. 346, Springer, Heidelberg, 2012. \MR{2954392}

\bibitem[Mor78]{Mo}
John~W Morgan, \emph{The algebraic topology of smooth algebraic varieties},
  Inst. Hautes \'Etudes Sci. Publ. Math. (1978), no.~48, 137--204.

\bibitem[NA87]{Na}
Vicente Navarro-Aznar, \emph{Sur la th\'eorie de {H}odge-{D}eligne}, Invent.
  Math. \textbf{90} (1987), no.~1, 11--76.

\bibitem[Pap85]{Papadima}
{\c{S}}tefan Papadima, \emph{The rational homotopy of {T}hom spaces and the
  smoothing of homology classes}, Comment. Math. Helv. \textbf{60} (1985),
  no.~4, 601--614. \MR{826873}

\bibitem[PS08]{PS}
Chris Peters and Joseph Steenbrink, \emph{Mixed {H}odge structures}, Ergebnisse
  der Mathematik und ihrer Grenzgebiete. 3. Folge. A Series of Modern Surveys
  in Mathematics, vol.~52, Springer-Verlag, Berlin, 2008.

\bibitem[Qui69]{Quillen}
Daniel Quillen, \emph{Rational homotopy theory}, Ann. of Math. (2) \textbf{90}
  (1969), 205--295. \MR{0258031}

\bibitem[RT00]{Rudyak-Tralle}
Yuli Rudyak and Aleksy Tralle, \emph{On {T}hom spaces, {M}assey products, and
  nonformal symplectic manifolds}, Internat. Math. Res. Notices (2000), no.~10,
  495--513. \MR{1759504}

\bibitem[Shi79]{Shiga}
Hiroo Shiga, \emph{Rational homotopy type and self-maps}, J. Math. Soc. Japan
  \textbf{31} (1979), no.~3, 427--434. \MR{535089}

\bibitem[Sul77]{Sullivan}
Dennis Sullivan, \emph{Infinitesimal computations in topology}, Inst. Hautes
  \'Etudes Sci. Publ. Math. (1977), no.~47, 269--331 (1978).

\bibitem[Tho54]{Thom:quelques}
Ren{\'e} Thom, \emph{Quelques propri\'et\'es globales des vari\'et\'es
  diff\'erentiables}, Comment. Math. Helv. \textbf{28} (1954), 17--86.

\bibitem[Voe98]{Voevodsky}
Vladimir Voevodsky, \emph{{$\mathbf A^1$}-homotopy theory}, Proceedings of the
  {I}nternational {C}ongress of {M}athematicians, {V}ol. {I} ({B}erlin, 1998),
  no. Extra Vol. I, 1998, pp.~579--604.

\end{thebibliography}

\end{document}